\documentclass[reqno,12pt]{amsart}

\usepackage{amsmath}
\usepackage{amssymb}
\usepackage{amsthm}
\usepackage[english]{babel}

\newtheorem{lemma}{Lemma}

\begin{document}
\begin{abstract}
The problem of the approximation of convolutions by accompanying
laws in the scheme of series satisfying the infinitesimality
condition is considered.    It is shown that the quality of approximation depends essentially on the
choice of centering constants.
\end{abstract}

\title[approximation of convolutions by
accompanying laws]{On the approximation of convolutions by
accompanying laws in the scheme of series}

\author{A. Yu. Zaitsev}

\address{St.~Petersburg Department of Steklov Mathematical Institute
\newline\indent
Fontanka 27, St.~Petersburg 191023, Russia\newline\indent
and St.~Petersburg State University}

\email{zaitsev@pdmi.ras.ru}

\keywords{triangular array of sums of independent
random vectors, infinitesimality condition
compound Poisson approximation}

\maketitle

The paper considers the problem of approximation of convolutions
by accompanying laws for a triangular array of sums of independent
random vectors satisfying an infinitesimality condition. We
introduce the necessary notation. We will denote by $ E_a $ the
distribution concentrated at a point $a\in \mathbf{R}^d$, $E=E_0$.
The compound Poisson distribution is given by
\begin{equation}
\label{2} e\left( \lambda F\right) =e^{-\lambda
}\sum_{s=0}^{\infty }\frac{\lambda ^{s}F^{s}}{s!},
\end{equation}
where $\lambda\geq0$, and $F$ is a probability distribution. Here
and below the products and powers of measures are understood in
the convolution sense. The L\'evy distance between one-dimensional
distributions  $ F $ and $ G $ is defined by
\begin{equation}
\label{25} L(F, G)
=\inf\big\{\varepsilon:F(x-\varepsilon)-\varepsilon\leq G(x)\leq
F(x+\varepsilon)+\varepsilon \hbox{ при всех
}x\in\mathbf{R}\big\},
\end{equation}
where  $F(\,\cdot\,)$ is the distribution function corresponding
to the distribution $F$.

The L\'evy--Prokhorov distance between the distributions $ F $ and
$ G $ in a complete separable metric space is defined as follows:
\begin{multline} \label{254} \pi(F, G)
=\inf\big\{\varepsilon:F\{X\}\leq
G\{X^\varepsilon\}+\varepsilon\hbox{ и }G\{X\}\leq
F\{X^\varepsilon\}+\varepsilon\\
\hbox{ for all Borel sets  }X\big\},
\end{multline}where $X^\varepsilon$ is the $\varepsilon$-neighborhood of $ X $.
It is well known that both L\'evy and
L\'evy--Prokhorov distances metrize the weak convergence of
probability distributions.

By the same letter $ c $ we denote positive absolute constants
which may be different even within a single formula. Writing $ A
\ll B $ means that $ A \leq c B $. If the corresponding constant
depends on the dimension $ d $, we will use the notation $ A \ll_d
B $. In the future,
 \,$\log ^\ast b=\max
\left\{ 1,\log b\right\} $, \,for \,$b>0$.

Consider the classical scheme of the series of sums of independent
random variables satisfying the infinitesimality condition (cf.
\cite{AZ, GK, P}). Let
$\{X_{j,k},j=1,2,\ldots;k=1,\ldots,n_{j}\}$, be independent random
variables with distributions
 $F_{j,k}=\mathcal
L(X_{j,k})$. Denote by
\begin{equation} \label{255}F_{j}=\prod_{k=1}^{n_{j}}F_{j,k},\quad
j=1,2,\ldots,
\end{equation}
the distributions of sums $S_{j}=\sum_{k=1}^{n_{j}}X_{j,k}$. The
infinitesimality condition is usually formulated as follows:
$$\max_{1\leq k\leq
n_{j}}\mathbf{P}\big\{\left|X_{j,k}\right|\geq\tau\big\}\to0
$$
 as
$j\to\infty$ for any $\tau>0$. For each fixed $j$, the
distributions of summands
 $X_{j,k}$ are connected with the distributions of summands
  $X_{m,s}$ при других $m\ne j$ only through the
infinitesimality condition. If under this condition the sequence
of distributions $ F_ j $ converges weakly to a probability
distribution~$D$ ($F_j\Rightarrow D$ as $j\to\infty$), then, by
Khinchin's theorem, the distribution ~$ D $ is infinitely
divisible. Conditions of convergence to a given infinitely
divisible distribution ~$D$ (see, e.g.,
 \cite{GK, P})
are usually formulated as conditions which are equivalent to the
convergence to the distribution ~$D$ of the so-called accompanying
infinitely divisible laws
\begin{equation}
\label{434}D_{j}=\prod_{k=1}^{n_{j}}\big(E_{a_{j,k}}\,
e\big(F_{j,k}\,E_{-a_{j,k}}\big)\big).\end{equation}
 The distributions $D_j$ depend on the centering constants
  $a_{j,k}$ which are defined by the equality
 \begin{equation}
\label{0}a_{j,k}=\int_{|x|\leq\tau}x
\,F_{j,k}\{dx\}.\end{equation} The constant $\tau$, involved in
the definition \eqref{0}, does not depend on $j$ and $k$ and it is
arbitrary. For different $\tau$, the numbers
 $a_{j,k}$,
are (generally speaking) different, and thus different are the
distribution~$D_j$ too. However, if, $F_j\Rightarrow D$, then also
$D_j\Rightarrow D$ as $j\to\infty$, for any $\tau$. It seems that
this would imply that the distribution of ~$D_j$ may be considered
as a good infinitely divisible approximation for the distributions
$F_j$, if the latter are defined through a scheme of series
satisfying the infinitesimality condition. However, this is not
always the case, at least if the accuracy of the approximation is
estimated in the L\'evy--Prokhorov metric. In particular, it may
be not so if the sequence of distributions of $ F_ j $ is not
relatively compact (in the topology of weak convergence). A
discussion of this circumstance is the subject of this paper.

The approximation of sequences of distributions that are not
relatively compact in the L\'evy--Prokhorov metric has a special
interest in connection with a recent result by Yu. A. Davydov and
V. I. Rotar'
 \cite{DR} on the characterization of sequences
distributions which are close each to other in the
L\'evy--Prokhorov metric in terms of the closeness of integrals of
uniformly continuous bounded functions.

First of all, note that the infinitesimality condition may be
reformulated as follows:
 \begin{equation} \label{73}\max_{1\leq
k\leq n_{j}}L(F_{j,k}, E)=\varepsilon_{j}\to0\quad\hbox{ as
}j\to\infty.\end{equation} This condition of closeness of a
distribution to the degenerate distribution was proposed by
Kolmogorov \cite{K1, K2} when considering the problem of the
infinitely divisible approximation of convolutions.
Condition~\eqref{73} is closely connected with the condition of
the representability of distributions $F_{j,k}$ in the form
\begin{equation}
\label{13} F_{j,k}=(1-p_{j,k})U_{j,k}+p_{j,k}V_{j,k},
\end{equation}
where $0\leq p_{j,k}\leq1$, the distributions  $U_{j,k}$,
$k=1,\ldots,n_{j}$, are concentrated on the segments $[-\tau_j,
\tau_j]$, $\tau_j\geq0$, $j=1,2,\ldots$, and $V_{j,k}$,
$k=1,\ldots,n_{j}$, are probability distributions, wherein
\begin{equation}
\label{43}\tau_j\to0\quad\hbox{and}\quad p_{j}=\max_{1\leq k\leq
n_{j}}p_{j,k}\to0\quad\hbox{ as }j\to\infty.
\end{equation}
This condition of closeness of a distribution to the degenerate
distribution was used by I.~A.~Ibragimov and E.~L. Presman
\cite{IP} when considering the above mentioned problem of
Kolmogorov.

It was shown that the natural infinitely divisible approximation
for distributions of
 $F_{j}=\prod_{k=1}^{n_{j}}F_{j,k}$ under conditions  \eqref{13} and \eqref{43}
 is given by the distributions
\begin{equation}
\label{436}G_{j}=\prod_{k=1}^{n_{j}}\big(E_{b_{j,k}}\,
e\big(F_{j,k}\,E_{-b_{j,k}}\big)\big).\end{equation}with
\begin{equation}
\label{03}b_{j,k}=\int_{\mathbf{R}}x
\,U_{j,k}\{dx\}.\end{equation}  It is easy to see that the
distributions \eqref{436} have the form  \eqref{434}, but with
replacing the $a_{j,k}$ by $b_{j,k}$. In \cite{ZA}, it was shown
that under conditions
 \eqref{13}, \eqref{43} and \eqref{03} we have
\begin{equation} \label{703}L(F_{j}, G_{j})\ll p_{j}+\tau_{j}\,
\log^{\ast}\tau_{j}^{-1}.
\end{equation}
and \begin{equation} \label{743}\pi(F_{j}, G_{j})\ll
\sum_{k=1}^{n_{j}}p_{j,k}^{2}+p_{j}+\tau_{j}\,\log^{\ast}\tau_{j}^{-1}.\end{equation}
If we assume additionally that the distributions  $V_{j,k}$ are
the same for all  $k=1,\ldots,n_{j}$, then
\begin{equation} \label{7431}\pi(F_{j}, G_{j})\ll
p_{j}+\tau_{j}\,\log^{\ast}\tau_{j}^{-1}.\end{equation}

The proofs of inequalities  \eqref{703}--\eqref{7431},  their
discussion and the history of the problem may be also found in the
monograph \cite{AZ}. Inequality \eqref{703} an optimal (with
respect to order) solution to the problem of Kolmogorov  \cite{K1,
K2} of infinitely divisible approximation of convolutions of
distributions satisfying condition
 \eqref{73}. Inequalities
\eqref{703}--\eqref{7431} are optimal with respect to order for
the dependence of the right-hand sides on $ p_ j $ and $ \tau_ j
$, and in general, the summand $\sum p_{j,k}^{2}$ can not be
removed from the the right-hand side of inequality
 \eqref{743}. Moreover, in general,
the right-hand sides of
 \eqref{703}--\eqref{7431} can not be
significantly reduced if the distribution of $ G_ j $ is replaced
by any other infinitely divisible distributions. Thus,
inequalities
 \eqref{703}--\eqref{7431} can be considered as
a quantitative refinement of classical Khinchin's theorem that the
limit distribution for the distribution $ F_ j $, defined in the
scheme of series of independent random variables satisfying the
infinitesimality condition, must be infinitely divisible, of
course, if it exists. In fact, if $ F_ j \Rightarrow D $, then, by
 \eqref{43} and \eqref{703}, we have the weak convergence
$ G_ j \Rightarrow D $ as $ j \to \infty $. The presence of the
limit distribution implies the relative compactness of the
sequence of distributions $ F_ j $. The distribution $ D $ is
infinitely divisible, as the limit of infinitely divisible
distributions. In the monograph \cite{AZ}, the conditions of
convergence of distributions $ F_ j $ to a given infinitely
divisible distribution $ D $ as $ j \to \infty $ are formulated as
conditions which are equivalent to the convergence to the
distribution ~ $ D $ of accompanying infinitely divisible laws $
G_ j $ with
 $b_{j,k}$, defined by equality  \eqref{03}.

Thus, if we study the  $F_{j}=\mathcal{L}(S_{j})$ and we are
interested in a reasonable infinitely divisible approximation for
the distributions~$ F_ j $, then, under conditions
 \eqref{13} and \eqref{43}, it is given by the distributions~$ G_ j $
 of the form \eqref{436} with $b_{j,k}$ from \eqref{03}. At the same time, if
$a_{j,k}$ are defined by  \eqref{0}, , the distributions $ F_ j $
and $ D_ j $ can be not close in the L\'evy--Prokhorov metric.
\medskip

\noindent\textbf{Example 1.} Consider as a simplest example the
distribution
\begin{equation}
\label{136} F_{j,k}=(1-j^{-1})E+j^{-1}E_1,\quad k=1,\ldots,n_{j}.
\end{equation}
Then conditions \eqref{13} and \eqref{43} are satisfied  with
$$U_{j,k}=E,\quad V_{j,k}=E_1,\quad p_{j,k}=p_{j}=j^{-1}, \quad\tau_{j}=0.$$
Furthermore, $b_{j,k}=0$ and and the distribution $ F_ j $  is
binomial with parameters  $n=n_{j}$ and $p=p_{j}=j^{-1}$. The
accompanying infinitely divisible distribution $ G_ j = e (n_j \,
p_ j \, E_1) $ is the Poisson distribution with parameter $ n_ j
\, p_ j $. According to a result of Yu.~V.~ Prokhorov  \cite{PR},
the distance in variation between distributions $ F_ j $ and $ G_
j $ is bounded from above by $c\,p_{j}=c\,j^{-1}$ and tends to
zero as  ${j\to\infty}$.

As to the approximating accompanying infinitely divisible
distribution $ D_ j $, then, as noted above, it depends on the
choice of the centering constants $a_{j,k}$. If $a_{j,k}$ are
chosen by the formula \eqref{0} with $\tau\geq1$, then
$a_{j,k}=j^{-1}$, $k=1,\ldots,n_{j}$, and \begin{eqnarray}
  D_{j} &=& \prod_{k=1}^{n_{j}}\big(E_{a_{j,k}}\,
e\big(F_{j,k}\,E_{-a_{j,k}}\big)\big) \nonumber\\
   &=& E_{n_{j}j^{-1}}\, e\big(n_{j}(1-j^{-1})\,E_{-j^{-1}}\big)\,
e\big(n_{j}j^{-1}\,E_{1-j^{-1}}\big).\label{483}
\end{eqnarray}Denote
$$
V_j=E_{n_{j}j^{-1}}\, e\big(n_{j}j^{-1}\,E_{1-j^{-1}}\big)
\quad\hbox{and}\quad W_j= e\big(n_{j}(1-j^{-1})\,E_{-j^{-1}}\big).
$$Тогда $D_{j} =V_j\,W_j$. Obviously,
\begin{eqnarray}
  D_{j}\big\{\mathbf{Z}^{1/8}\big\} &=&
  \int_{\mathbf{R}}\Big(\int_{\mathbf{R}}\mathbf{1}\big\{x+y\in\mathbf{Z}^{1/8}\big\}
  W_j\{dx\}\Big)\,V_j\{dy\}\nonumber\\
   &\leq& \sup_{x\in\mathbf{R}}W_j\big\{\mathbf{Z}^{1/8}+x\big\}.\label{4843}
\end{eqnarray}Here
$\mathbf{1}\{\,\cdot\,\}$ is the indicator function, and
$\mathbf{Z}^{1/8}$ means the $ 1 /  8$-neighborhood of the set of
all integers $ \mathbf Z $. The set  $\mathbf{Z}^{1/8}$ consists
of real numbers $y\in\mathbf{R}$ representable as
 $y=z+t$,
where $z\in\mathbf{Z}$ and $|t|<1/8$.
\medskip

Let $U_{\lambda}$ be the Poisson distribution with parameter $
\lambda> 0 $, and $ \xi_ \lambda $ be a random variable with
distribution $ U_ \lambda $. Then
\begin{equation}
\label{139}
\mathbf{P}\big\{\xi_{\lambda}=s\big\}=\frac{e^{-\lambda}\lambda^{s}}{s!},\quad
s=0,1,2,\ldots,
\end{equation}
and $\mathbf{E}\,\xi_{\lambda}=\mathbf{D}\,\xi_{\lambda}=\lambda$.

In order to estimate the right-hand side of  \eqref{4843}, we need
the following elementary property of the Poisson distribution.

\begin{lemma} \label{lms}
There exist absolute positive constants $ c_1 $ and $ c_2 $ such
that
\begin{equation}\label{4669}
\sup_{x\in\mathbf{R}}\mathbf{P}\big\{\delta^{-1}\xi_{\lambda}\in\mathbf{Z}^{1/8}+x\big\}\leq5/8,\quad
\hbox{if } \delta\geq c_1\hbox{ and } \ \lambda\geq
c_2\,\delta^{2}.
\end{equation}
\end{lemma}
Choosing constants $ 1/8 $ and a $5 / 8$ in the statement of the
lemma is rather arbitrary, it is important that they are some
absolute constants satisfying $ 0 < 1/8 < 1/4 $ and $ 5/8 < 1$.

\begin{proof}[Proof of Lemma\/ \ref{lms}] Let $x\in\mathbf{R}$.
The set $$A_x=\delta\,(\mathbf{Z}^{1/8}+x)$$ is a union of open
intervals  $ I_j $ of length $ \delta / 4$ with centers at the
points $ \delta(j + x) $, $ j \in \mathbf Z $. Therefore,
\begin{equation} \label{619}\mathbf{P}\big\{\xi_{\lambda}\in
A_x\big\} =\sum_{j\in\mathbf{Z}}\mathbf{P}\big\{\xi_{\lambda}\in
I_j\big\}.
\end{equation}
Accordingly, the set $ \mathbf R \setminus A_x $ consists of
closed intervals $ J_j $ of length  $ 3\delta / 4$, located
between intervals of the set
 $A_x$, and \begin{equation}
\label{659}\mathbf{P}\big\{\xi_{\lambda}\in \mathbf{R}\setminus
A_x\big\} =\sum_{j\in\mathbf{Z}}\mathbf{P}\big\{\xi_{\lambda}\in
J_j\big\}.\end{equation} We assume that the segments $ J_j $ are numbered in
such a way that for each $ j $ the interval $ J_j $ is located
directly behind the interval $ I_j $, if we move on the real axis
in the direction  $ + \infty $.

Obviously,
\begin{equation}
\label{439}
\mathbf{P}\big\{\xi_{\lambda}=s\big\}\geq\mathbf{P}\big\{\xi_{\lambda}=s+1\big\},\quad
\hbox{если } s+1\geq \lambda,
\end{equation}
и \begin{equation} \label{4889}
\mathbf{P}\big\{\xi_{\lambda}=s\big\}\geq\mathbf{P}\big\{\xi_{\lambda}=s-1\big\},\quad
\hbox{if } \lambda\geq s.
\end{equation}
Choosing the constant $ c_1 $ being large enough, we can ensure
that the number of lattice points in each interval $ I_j $ will be
less than the number of lattice points in each of the segments $
J_m $, $ m \in \mathbf Z $. Let $ I_k $ and $I_{k+1}$ be two
intervals $ I_j $, located closest to the point $ \lambda $. If
the second and third by proximity to the point $ \lambda $
intervals $ I_j $ are from it at the same distance, for
definiteness we take one of them, which is located between the $
\lambda $ and $ + \infty $. Inequalities
 \eqref{439}--\eqref{4889} imply that
\begin{equation} \label{3659}
\mathbf{P}\big\{\xi_{\lambda}\in I_j\big\}\leq
\mathbf{P}\big\{\xi_{\lambda}\in J_{j-1}\big\},\quad \hbox{if }
j\geq k+2,
\end{equation}
\begin{equation} \label{8859}
\mathbf{P}\big\{\xi_{\lambda}\in I_j\big\}\leq
\mathbf{P}\big\{\xi_{\lambda}\in J_{j}\big\},\quad \hbox{if }
j\leq k-1.
\end{equation}

It is well known that \begin{equation} \label{4219}
\max_{k\in\mathbf{Z}}\mathbf{P}\big\{\xi_{\lambda}=k\big\}\gg{\lambda}^{-1/2},\quad
\hbox{if } \lambda\geq 1.
\end{equation}
In order to prove  \eqref{4219} it  is sufficient to note that
 $$U_{\lambda}=U_{\lambda/n}^{n}\quad \hbox{for all positive integers  } n$$
 and and to use the Berry--Ess\'een inequality
(see
 \cite{AZ, P}). It is even easier to verify the validity of  \eqref{4219},
 applying the Stirling formula to the right-hand side of
\eqref{139} and taking into account  \eqref{439}--\eqref{4889}.

Inequality \eqref{4219} implies that
\begin{equation} \label{3219}
\max_{j\in\mathbf{Z}}\mathbf{P}\big\{\xi_{\lambda}\in
I_j\big\}\leq c\,\delta\,{\lambda}^{-1/2}\leq1/8,
\end{equation}if the constant $ c_2 $ is large enough.

Substituting inequalities  \eqref{3659}--\eqref{8859} in the
equality \eqref{619} and using relations  \eqref{659} and
\eqref{3219}, we get
\begin{equation} \label{6179}\mathbf{P}\big\{\xi_{\lambda}\in
A_x\big\} \leq1/4+\mathbf{P}\big\{\xi_{\lambda}\in
\mathbf{R}\setminus A_x\big\}=5/4-\mathbf{P}\big\{\xi_{\lambda}\in
A_x\big\}.
\end{equation}This implies inequality  \eqref{4669}. 
\end{proof}

Let us return to the evaluation of the right-hand side of
\eqref{4843}. It is easy to see that
$W_j=\mathcal{L}(-\delta_{j}^{-1}\,\xi_{\lambda_{j}})$ with
\begin{equation} \label{1309}
\delta_{j}=j,\quad \lambda_{j}=n_{j}(1-j^{-1}).
\end{equation}
By choosing  $n_{j}\geq2\,c_2\,j^{2}$, and by applying Lemma
\ref{lms} with $\delta=\delta_{j}$, $\lambda=\lambda_{j}$, we
obtain that, for ${j}\geq\max\big\{2,c_1\big\}$,
\begin{eqnarray}
 \sup_{x\in\mathbf{R}}W_j\big\{\mathbf{Z}^{1/8}+x\big\}  &=&
 \sup_{x\in\mathbf{R}}\mathbf{P}\big\{-\delta_{j}^{-1}\,\xi_{\lambda_{j}}
 \in\mathbf{Z}^{1/8}+x\big\} \nonumber\\
   &=& \sup_{x\in\mathbf{R}}\mathbf{P}\big\{\delta_{j}^{-1}\,\xi_{\lambda_{j}}
   \in\mathbf{Z}^{1/8}+x\big\}         \leq5/8 . \label{4759}
\end{eqnarray}
Inequalities \eqref{4843} and \eqref{4759} imply that
\begin{equation} \label{1009}
D_{j}\big\{\mathbf{Z}^{1/8}\big\}\leq5/8.
\end{equation}
At the same time,
\begin{equation} \label{1001}
F_{j}\big\{\mathbf{Z}\big\}=1.
\end{equation}
According to the definition  \eqref{254}, relations \eqref{1009}
and \eqref{1001} imply that \begin{equation} \label{1005}
\pi\big(F_{j},D_j\big)\geq1/8.
\end{equation}
Thus, the L\'evy--Prokhorov distance $ \pi \big (F_ j, D_j \big) $
does not tend to zero as $ j \to \infty $, in contrast to the $
\pi \big (F_ j, G_j \big) $. Of course, this is due to the fact
that the sequence of distributions $ F_ j $ is not relatively
compact.
\medskip

Inequalities \eqref{703}--\eqref{7431} and Example 1 allow us to
conclude that, in the case of identically distributed summands,
the distributions $ G_j $ may be always regarded as a good
approximation for the distributions $ F_ j $, while the
distributions $ D_j $ may be far from the distributions $ F_ j $
in the L\'evy--Prokhorov metric. Note that in some cases we have
$a_{j,k}=b_{j,k}$ and $D_j=G_j$. For example, if all considered
distributions are symmetric. Or if, in just considered Example 1,
 $\tau<1$ and $a_{j,k}=b_{j,k}=0$.\medskip

It is important here is that the distributions $ G_j $ are defined
for $ \tau_ j \to0 $, so that inequalities
 \eqref{703} and \eqref{7431}
ensure the closeness of distributions $G_j$ and $F_{j}$. At the
same time, in the definition
 \eqref{0} of $a_{j,k}$ the value of $ \tau $
is fixed and does not depend on $ j $.

Another important difference between  $b_{j,k}$ and $a_{j,k}$ is
that even if the distributions
 $U_{j,k}$ and $V_{j,k}$
are concentrated, respectively, on the intervals $ [- \tau_j,
\tau_j] $ and on the sets $ \mathbf R \setminus [- \tau_j, \tau_j]
$, then, for  $ \tau = \tau_j $, we have the equality
$a_{j,k}=(1-p_{j,k})\,b_{j,k}$. The presence of the factor
$(1-p_{j,k})$ in this equality leads to the fact that
distributions $ D_j $ and $ F_ j $ may be far from each other in
the L\'evy--Prokhorov metric, even if in the definition
 \eqref{0} of $a_{j,k}$ the
value of $ \tau = \tau_j $ depends on $ j $ and tends to zero as $
j \to \infty $. This is illustrated by the following Example 2
which is a modification of Example 1.

\medskip

\noindent\textbf{Example 2.} Now let
\begin{equation}
\label{5136}
F_{j,k}=(1-j^{-1})E_{-j^{-1}}+j^{-1}E_{1-j^{-1}},\quad
k=1,\ldots,n_{j}.
\end{equation}
Then conditions  \eqref{13} and \eqref{43} are satisfied  with
$$
U_{j,k}=E_{-j^{-1}},\quad V_{j,k}=E_{1-j^{-1}},\quad
p_{j,k}=p_{j}=j^{-1},\quad \tau_{j}=j^{-1}.$$ Furthermore,
$b_{j,k}=-j^{-1}$ and the distribution  $F_{j}E_{n_{j}j^{-1}}$ --
is binomial with parameters  $n_{j}$ and $j^{-1}$. Hence, the
distribution $ F_ j $ itself is a binomial distribution shifted by
$n_{j}j^{-1}$ to the left. The accompanying infinitely divisible
distribution
$$
G_{j}=E_{-n_{j}j^{-1}}e(n_j\,p_{j}\,E_1)
$$
is a similarly shifted Poisson distribution with parameter $ n_ j
\, p_ j $. In analogy to Example 1, the distance in variation
between the distributions $ F_ j $ and $ G_ j $ is bounded from
above by  $c\,p_{j}=c\,j^{-1}$ and tends to zero as
${j\to\infty}$.

As to the approximating accompanying infinitely divisible
distribution $ D_ j $, then if the values of   $a_{j,k}$ are
chosen by the formula \eqref{0} with $\tau=j^{-1}$, $j\geq3$, then
$a_{j,k}=-j^{-1}(1-j^{-1})$, $k=1,\ldots,n_{j}$, and
\begin{eqnarray}
  D_{j} &=& \prod_{k=1}^{n_{j}}\big(E_{a_{j,k}}\,
e\big(F_{j,k}\,E_{-a_{j,k}}\big)\big) \nonumber\\
   &=& E_{-n_{j}j^{-1}(1-j^{-1})}\, e\big(n_{j}(1-j^{-1})\,E_{-j^{-2}}\big)\,
e\big(n_{j}j^{-1}\,E_{1-j^{-2}}\big).\label{5483}
\end{eqnarray}
By choosing  $n_{j}\geq2\,c_2\,j^{4}$ and proceeding by analogy
with Example~1, it is easy to show that, for
${j}\geq\max\big\{3,c_1\big\}$ the L\'evy--Prokhorov distance $
\pi \big (F_ j, D_j \big) $ is separated from the zero and,
therefore, does not tend to zero as $ j \to \infty $, in contrast
to $ \pi \big (F_ j, G_j \big) $.

\medskip

In the multidimensional case the situation does not differ from
one-dimensional one. Let now
$\{X_{j,k},j=1,2,\ldots;k=1,\ldots,n_{j}\}$ be independent $d
$-dimensional random vectors with distributions  $F_{j,k}=\mathcal
L(X_{j,k})$, representable as
\begin{equation}
\label{1344} F_{j,k}=(1-p_{j,k})U_{j,k}+p_{j,k}V_{j,k},
\end{equation}
where $0\leq p_{j,k}\leq1$, the distributions  $U_{j,k}$,
$k=1,\ldots,n_{j}$, are concentrated on the Euclidean balls
$B_{\tau_j}$, centered at the origin and of radii  $\tau_j\geq0$,
$j=1,2,\ldots$,  $V_{j,k}$, $k=1,\ldots,n_{j}$, are arbitrary
distributions, and
\begin{equation}
\label{4344}\tau_j\to0\quad\hbox{and}\quad p_{j}=\max_{1\leq k\leq
n_{j}}p_{j,k}\to0\quad\hbox{ as }j\to\infty.
\end{equation} Let the distributions $ F_ j $ and $ G_j $ be defined by
 \eqref{255} and \eqref{436}, where
\begin{equation}
\label{035}b_{j,k}=\int_{\mathbf{R}^{d}}x
\,U_{j,k}\{dx\}.\end{equation} In author's paper  \cite{Z}, it was
shown that, under the conditions
 \eqref{1344}, \eqref{4344}
and \eqref{035}, we have
\begin{equation} \label{7032}L(F_{j}, G_{j})\ll_{d} p_{j}+\tau_{j}\,
\log^{\ast}\tau_{j}^{-1}.
\end{equation}
and \begin{equation} \label{7434}\pi(F_{j}, G_{j})\ll_{d}
\sum_{k=1}^{n_{j}}p_{j,k}^{2}+p_{j}+\tau_{j}\,\log^{\ast}\tau_{j}^{-1}.\end{equation}
If we assume additionally that the distributions  $V_{j,k}$ are
the same for all  $k=1,\ldots,n_{j}$, then
\begin{equation} \label{7435}\pi(F_{j}, G_{j})\ll_{d}
p_{j}+\tau_{j}\,\log^{\ast}\tau_{j}^{-1}.\end{equation} The
multidimensional analog of the L\'evy distance is defined just as
the L\'evy--Prokhorov distance, only the Borel sets should be
replaced by the parallelepipeds.

If, in addition to the infinitesimality condition \eqref{4344}, we
assume that
$$
\sum_{k=1}^{n_{j}}p_{j,k}^{2}\to0 \quad\hbox{as }j\to\infty,
$$
then $ \pi \big (F_ j, G_j \big) $ tends to zero. If the
distributions
 $V_{j,k}$ are the same for all
$ k = 1 , \ldots, n_ j $, then for $ \pi \big (F_ j, G_j \big)
\to0$ no additional assumptions are required.

In \cite{Z2}, one can find the estimates for the accuracy of
strong approximation of the corresponding random vectors which
follow from the estimates of the L\'evy--Prokhorov distance
 \eqref{7434} and \eqref{7435}.

\medskip

Yu.~A.~Davydov and V.~I.~ Rotar' \cite {DR} established, in
particular, that the sequences $ d $-di\-men\-sional distributions
approach each other in the L\'evy--Prokhorov metric if and only if
the integrals over these distributions of uniformly bounded
continuous functions do the same. It is clear that the above
distributions $ F_ j $ and $ G_j $ of a triangular array
satisfying the infinitesimality condition, provide a large number
of meaningful examples of approaching each other sequences of
distributions, including those which are not relatively compact.
\medskip

{\bf Acknowledgment} Research supported by grants
NSh-1216.2012.01, RFBR 10-01-00242 and 11-01-12104,  and by a
program of fundamental researches of Russian Academy of Sciences
``Modern problems of fundamental mathematics''.

\end{document}